\documentclass{amsart}
\usepackage{amsfonts}

\newtheorem{theorem}{Theorem}[section]

\newtheorem{lemma}[theorem]{Lemma}

\newtheorem{proposition}[theorem]{Proposition}
\newtheorem{remark}[theorem]{Remark}
\newtheorem{comments }[theorem]{Comments}
\numberwithin{theorem}{section}
\input{tcilatex}

\begin{document}
\title[Crossed products by Hilbert pro-$C^{\ast }$-bimodules versus tensor
products]{Crossed products by Hilbert pro-$C^{\ast }$-bimodules versus
tensor products }
\author{Maria Joi\c{t}a}
\address{Maria Joi\c{t}a \\
Department of Mathematics, Faculty of Applied Sciences, University
Politehnica of Bucharest, 313 Spl. Independentei, 060042, Bucharest, Romania
and Simion Stoilow Institute of Mathematics of the Roumanian Academy, 21
Calea Grivitei Street, 010702 Bucharest, Romania}
\email{mjoita@fmi.unibuc.ro }
\urladdr{http://sites.google.com/a/g.unibuc.ro/maria-joita/}
\thanks{This paper is in final form and no version of it will be submitted
for publication elsewhere.}

\begin{abstract}
We show that if $(X.A)$ and $(Y,B)$ are two isomorphic Hilbert pro-$C^{\ast
} $-bimodules, then the crossed product $A\times _{X}\mathbb{Z}$ of $A$ by $%
X $ and the crossed product $B\times _{Y}\mathbb{Z}$ of $B$ by $Y$ are
isomorphic as pro-$C^{\ast }$-algebras. We also prove a property of
"associativity" between " $\otimes _{\min }$" and "$\times _{X}$" $\ $as
well as " $\otimes _{\max }$" and "$\times _{X}$". As an application of
these results we show that the crossed product of a nuclear pro-$C^{\ast }$%
-algebra $A$ by a full Hilbert pro-$C^{\ast }$-bimodule $X$ is a nuclear pro-%
$C^{\ast }$-algebra.
\end{abstract}

\maketitle

\section{Introduction}

The crossed product $A\times _{\alpha }\mathbb{Z}$ of a $C^{\ast }$-algebra $%
A$ by an automorphism $\alpha $ is isomorphic with the $C^{\ast }$-algebra
generated by the spectral subspaces $\left( A\times _{\alpha }\mathbb{Z}%
\right) _{0}$ and $\left( A\times _{\alpha }\mathbb{Z}\right) _{1}$ of the
dual action of the unit circle on $A\times _{\alpha }\mathbb{Z}$ induced by $%
\alpha $. A natural question is when given a $C^{\ast }$-algebra $B$ and an
action of the unit circle on $B$, there is a $C^{\ast }$-algebra $B_{0}$ and
an automorphism $\alpha _{0}$ of $B_{0}$ such that $B$ is isomorphic to the
crossed product of $B_{0}$ by $\alpha _{0}$. To answer to this question, B.
Abadie, S. Eilers and R. Exel \cite{AEE} introduced the notion of crossed
product by Hilbert $C^{\ast }$-bimodules that is a generalization of the
notion of crossed product by automorphisms. Here, the automorphism $\alpha $
of $A$ is replaced by a Hilbert $C^{\ast }$-bimodule $X$ over $A$. It is
known that the crossed product of $C^{\ast }$-algebras by automorphisms is
invariant under conjugacy, and the minimal tensor product $\left( A\times
_{\alpha }\mathbb{Z}\right) \otimes _{\min }B$ of the crossed product $%
A\times _{\alpha }\mathbb{Z}$ of $A$ by $\alpha $ is isomorphic to the
crossed product of $A\otimes _{\min }B$ by the automorphism $\alpha \otimes
_{\min }$id$_{B}$. It is natural to ask if these results can be extended in
the context of crossed products by Hilbert $C^{\ast }$-bimodules. Crossed
products of pro-$C^{\ast }$-algebras by inverse limit actions of locally
compact groups were first considered by N.C. Philips \cite{P1} and secondly
by M. Joi\c{t}a \cite{J4}. In \cite{JZ}, M. Joi\c{t}a and I. Zarakas
extended the construction of B. Abadie, S. Eilers and R. Exel in the context
of pro-$C^{\ast }$-algebras.

In this paper we prove that if two Hilbert (pro-)$C^{\ast }$-bimodules are
isomorphic, then the crossed products associated are isomorphic. Also we
show that given a Hilbert (pro-)$C^{\ast }$-bimodule $(X,A)\ $and a (pro-)$%
C^{\ast }$-algebra $B$, the crossed product associated to $(X\otimes _{\min
}B,A\otimes _{\min }B)$ is isomorphic to the minimal tensor product of the
crossed product associated to $(X,A)$ and $B$, and the crossed product
associated to $(X\otimes _{\max }B,A\otimes _{\max }B)$ is isomorphic to the
maximal tensor product of the crossed product associated to $(X,A)$ and $B$.
As an application of these results we show that the crossed product $A\times
_{X}\mathbb{Z\ }$of a nuclear pro-$C^{\ast }$-algebra $A$ by a full Hilbert
pro-$C^{\ast }$-bimodule $X$ is a nuclear pro-$C^{\ast }$-algebra.

\section{Preliminaries}

A \textit{pro-}$C^{\ast }$\textit{-algebra} is a complete Hausdorff
topological $\ast $-algebra $A$\ whose topology is given by a directed
family of $C^{\ast }$-seminorms $\{p_{\lambda };\lambda \in \Lambda \}$.

Throughout this paper all vector spaces are considered over the field $%
\mathbb{C}$ of complex numbers and all topological spaces are assumed to be
Hausdorff. Also the topology on the pro-$C^{\ast }$-algebras $A\ $and $B$ is
given by the family of $C^{\ast }$-seminorms $\Gamma =\{p_{\lambda };\lambda
\in \Lambda \}$, respectively $\Gamma ^{\prime }=\{q_{\delta };\delta \in
\Delta \}.$

A \textit{pro-}$C^{\ast }$\textit{-morphism} is a continuous $\ast $%
-morphism $\varphi :A\rightarrow B$ (that is, $\varphi $ is linear, $\varphi
\left( ab\right) =\varphi (a)\varphi (b)$ or all $a,b\in A,$ $\varphi
(a^{\ast })=\varphi (a)^{\ast }$ for all $a\in A$ and for each $q_{\delta
}\in \Gamma ^{\prime }$, there is $p_{\lambda }\in \Gamma $ such that $%
q_{\delta }\left( \varphi (a)\right) \leq p_{\lambda }\left( a\right) $ for
all $a\in A$). An invertible pro-$C^{\ast }$-morphism $\varphi :A\rightarrow
B$ is a \textit{pro-}$C^{\ast }$\textit{-isomorphism} if $\varphi ^{-1}$ is
also pro-$C^{\ast }$-morphism.

For $\lambda \in \Lambda $,\ $\ker p_{\lambda }$\ is a closed $\ast $%
-bilateral ideal and $A_{\lambda }=A/\ker p_{\lambda }$\ is a $C^{\ast }$%
-algebra in the $C^{\ast }$-norm $\left\Vert \cdot \right\Vert _{p_{\lambda
}}$\ induced by $p_{\lambda }$\ (that is, $\left\Vert a+\ker p_{{\small %
\lambda }}\right\Vert _{p_{\lambda }}=$\ $p_{{\small \lambda }}(a),$ for all 
$a\in A$). The canonical map from $A$ to $A_{\lambda }$ is denoted by $\pi
_{\lambda }^{A},$ $\pi _{\lambda }^{A}\left( a\right) =a+\ker p_{\lambda }$
for all $a\in A$. For $\lambda ,\mu \in \Lambda $\ with $\mu \leq \lambda $\
there is a surjective $C^{\ast }$-morphism $\pi _{\lambda \mu
}^{A}:A_{\lambda }\rightarrow A_{\mu }$\ such that $\pi _{\lambda \mu
}^{A}\left( a+\ker {\small p}_{\lambda }\right) =a+\ker p_{\mu }$, and then $%
\{A_{\lambda };\pi _{\lambda \mu }^{A}\}_{\lambda ,\mu \in \Lambda }$\ is an
inverse system of $C^{\ast }$-algebras. Moreover, pro-$C^{\ast }$-algebras$\
A$ and $\lim\limits_{\leftarrow \lambda }A_{\lambda }$ are isomorphic (the
Arens-Michael decomposition of $A$). For more details we refer the reader 
\cite{F,P}.

Here we recall some basic facts from \cite{J1} and \cite{Z} regarding
Hilbert pro-$C^{\ast }$-modules and Hilbert pro-$C^{\ast }$-bimodules
respectively.

\textit{A} \textit{right Hilbert pro-}$C^{\ast }$\textit{-module over }$A$%
\textit{\ }(or just \textit{Hilbert }$A$\textit{-module}), is a linear space 
$X$ that is also a right $A$-module equipped with a right $A$-valued inner
product $\left\langle \cdot ,\cdot \right\rangle _{A}$, that is $\mathbb{C}$%
- and $A$-linear in the second variable and conjugate linear in the first
variable, with the following properties:

\begin{enumerate}
\item $\left\langle x,x\right\rangle _{A}\geq 0$ and $\left\langle
x,x\right\rangle _{A}=0$ if and only if $x=0;$

\item $\left( \left\langle x,y\right\rangle _{A}\right) ^{\ast
}=\left\langle y,x\right\rangle _{A}$
\end{enumerate}

and which is complete with respect to the topology given by the family of
seminorms $\{p_{\lambda }^{A}\}_{\lambda \in \Lambda },$ with $p_{\lambda
}^{A}\left( x\right) =p_{\lambda }\left( \left\langle x,x\right\rangle
_{A}\right) ^{\frac{1}{2}},x\in X$. A Hilbert $A$-module is \textit{full} if
the pro-$C^{\ast }$- subalgebra of $A$ generated by $\{\left\langle
x,y\right\rangle _{A};x,y\in X\}$ coincides with $A$.

\textit{A} \textit{left Hilbert pro-}$C^{\ast }$\textit{-module }$X$\ over\
a\ pro-$C^{\ast }$-algebra $A$ is defined in the same way, where for
instance the completeness is requested with respect to the family of
seminorms $\{^{A}p_{\lambda }\}_{\lambda \in \Lambda }$, where $%
^{A}p_{\lambda }\left( x\right) =p_{\lambda }\left( _{A}\left\langle
x,x\right\rangle \right) ^{\frac{1}{2}},x\in X$.

In case $X$ is a left Hilbert \text{pro-}$C^{\ast }$\text{-}module over $A$
and a right Hilbert \text{pro-}$C^{\ast }$\text{-}module over $B$, the
topology on $B$ is given by the family of $C^{\ast }$-seminorms $%
\{q_{\lambda }\}_{\lambda \in \Lambda }$, such that the following relations
hold:

\begin{itemize}
\item $_{A}\left\langle x,y\right\rangle z=x\left\langle y,z\right\rangle
_{B}$ for all $x,y,z\in X$ ,

\item $q_{\lambda }^{B}(ax)$ $\leq p_{\lambda }(a)q_{\lambda }^{B}\left(
x\right) $ and $^{A}p_{\lambda }(xb)$ $\leq q_{\lambda }(b)^{A}p_{\lambda
}\left( x\right) $ for all $x\in X,\,a\in A,\,b\in B$ and for all $\lambda
\in \Lambda $,
\end{itemize}

then we say that $X$ is \textit{a Hilbert }$A-B$\textit{\ pro-}$C^{\ast }$%
\textit{-bimodule}.

A Hilbert $A-B$ pro-$C^{\ast }$\text{-}bimodule $X$ is \textit{full} if it
is full as a right and as a left Hilbert pro-$C^{\ast }$-module. Clearly,
any pro-$C^{\ast }$-algebra $A$ has a canonical structure of Hilbert $A-A$
bimodule with the inner products given by $_{A}\left\langle a,b\right\rangle 
$ $=ab^{\ast }$, respectively $\left\langle a,b\right\rangle _{A}$ $=a^{\ast
}b$ for all $a,b\in A.$

\textit{A morphism of Hilbert pro-}$C^{\ast }$\textit{-bimodules }from a
Hilbert $A-A$\ pro-$C^{\ast }$-bimodule $X$ to a Hilbert $B-B$\textit{\ }pro-%
$C^{\ast }$-bimodule $Y$ is a pair $\left( \Phi ,\varphi \right) $
consisting of a pro-$C^{\ast }$-morphism $\varphi :A\rightarrow B$ and a map 
$\Phi :X\rightarrow Y,\ $such that the following relations hold:

\begin{itemize}
\item $_{B}\left\langle \Phi \left( x\right) ,\Phi \left( y\right)
\right\rangle =\varphi \left( _{A}\left\langle x,y\right\rangle \right) $
and $\Phi \left( ax\right) =\varphi \left( a\right) \Phi \left( x\right) $
for all $x,y\in X,$ $a\in A,$

\item $\left\langle \Phi \left( x\right) ,\Phi \left( y\right) \right\rangle
_{B}=\varphi \left( \left\langle x,y\right\rangle _{A}\right) $ and $\Phi
\left( xa\right) =\Phi \left( x\right) \varphi \left( a\right) $ for all $%
x,y\in X,$ $a\in A.$
\end{itemize}

A morphism of Hilbert pro-$C^{\ast }$-bimodules\textit{\ }$\left( \Phi
,\varphi \right) :(X,A)\rightarrow (Y,B)$ is \textit{an isomorphism of
Hilbert pro-}$C^{\ast }$\textit{-bimodules} if\textit{\ }$\Phi $ and $%
\varphi $ are invertible and $\left( \Phi ^{-1},\varphi ^{-1}\right)
:(Y,B)\rightarrow (X,A)$ is a morphism of Hilbert pro-$C^{\ast }$-bimodules.

Let $X$ be a Hilbert $A-A$ pro-$C^{\ast }$-bimodule. Then, for each $\lambda
\in \Lambda ,$ $^{A}p_{\lambda }\left( x\right) =p_{\lambda }^{A}\left(
x\right) $ for all $x\in X$, and the normed space $X_{\lambda }=X/\ker
p_{\lambda }^{A}$, where $\ker p_{\lambda }^{A}=\{x\in X;p_{\lambda
}^{A}\left( x\right) =0\}$, is complete in the norm $||x+\ker p_{\lambda
}^{A}||_{X_{\lambda }}$ $=p_{\lambda }^{A}(x),x\in X$. Moreover, $X_{\lambda
}$ has a canonical structure of a Hilbert $A_{\lambda }-$ $A_{\lambda }$ $%
C^{\ast }$-bimodule with $\left\langle x+\ker p_{\lambda }^{A},y+\ker
p_{\lambda }^{A}\right\rangle _{A_{\lambda }}$ $=\left\langle
x,y\right\rangle _{A}+\ker p_{\lambda }$ and $_{A_{\lambda }}\left\langle
x+\ker p_{\lambda }^{A},y+\ker p_{\lambda }^{A}\right\rangle =\
_{A}\left\langle x,y\right\rangle +\ker p_{\lambda }$ for all $x,y\in X$.
The canonical surjection from $X$ to $X_{\lambda }$ is denoted by $\sigma
_{\lambda }^{X}$. For $\lambda ,\mu \in \Lambda $ with $\lambda \geq \mu $,
there is a canonical surjective linear map $\sigma _{\lambda \mu
}^{X}:X_{\lambda }\rightarrow X_{\mu }$ such that $\sigma _{\lambda \mu
}^{X}\left( x+\ker p_{\lambda }^{A}\right) =x+\ker p_{\mu }^{A}$ for all $%
x\in X$. For $\lambda ,\mu \in \Lambda $ with $\lambda \geq \mu ,$ $\left(
\sigma _{\lambda \mu }^{X},\pi _{\lambda \mu }^{A}\right) \ $is a morphism
of Hilbert $C^{\ast }$-bimodule and $\lim\limits_{\leftarrow \lambda
}X_{\lambda }$ has a canonical structure of Hilbert $\lim\limits_{\leftarrow
\lambda }A_{\lambda }-\lim\limits_{\leftarrow \lambda }A_{\lambda }$
bimodule. Moreover, $X=$ $\lim\limits_{\leftarrow \lambda }X_{\lambda }$, up
to an isomorphism of Hilbert pro-$C^{\ast }$- bimodules.

\textit{A covariant representation }of a Hilbert pro-$C^{\ast }$-bimodule $%
\left( X,A\right) $ on a pro-$C^{\ast }$-algebra $B$ is a morphism of
Hilbert pro-$C^{\ast }$-bimodules from $\left( X,A\right) $ to the Hilbert
pro-$C^{\ast }$-bimodule $\left( B,B\right) $.

\textit{The crossed product of }$A$\textit{\ }by\textit{\ }a Hilbert pro-$%
C^{\ast }$-bimodule $\left( X,A\right) $\textit{\ }is a pro-$C^{\ast }$%
-algebra, denoted by $A\times _{X}\mathbb{Z}$, and a covariant
representation $\left( i_{X},i_{A}\right) $ of $\left( X,A\right) $ on $%
A\times _{X}\mathbb{Z}$ with the property that for any covariant
representation $\left( \varphi _{X},\varphi _{A}\right) $ of $\left(
X,A\right) $ on a pro-$C^{\ast }$-algebra $B$, there is a unique pro-$%
C^{\ast }$-morphism $\Phi :A\times _{X}\mathbb{Z}\rightarrow B$ such that $%
\Phi \circ i_{X}=\varphi _{X}$ and $\Phi \circ i_{A}=\varphi _{A}$ \cite[%
Definition 3.3]{JZ}.

\begin{remark}
If $\left( \Phi ,\varphi \right) $ is a morphism of Hilbert pro-$C^{\ast }$%
-bimodules from $\left( X,A\right) $ to $\left( Y,B\right) $, then $\left(
i_{Y}\circ \Phi ,i_{B}\circ \varphi \right) $ is a covariant representation
of $\left( X,A\right) $ on $B\times _{Y}\mathbb{Z}$ and by the universal
property of $A\times _{X}\mathbb{Z}$ there is a unique pro-$C^{\ast }$%
-morphism $\Phi \times \varphi $ from $A\times _{X}\mathbb{Z}$ to $B\times
_{Y}\mathbb{Z}$ such that $\left( \Phi \times \varphi \right) \circ
i_{A}=i_{B}\circ \varphi $ and $\left( \Phi \times \varphi \right) \circ
i_{X}=i_{Y}\circ \Phi $.
\end{remark}

\section{Crossed products by Hilbert bimodules and tensor products}

An automorphism $\alpha $ of a pro-$C^{\ast }$-algebra $A$ such that $%
p_{\lambda }(\alpha (a))=p_{\lambda }(a)$ for all $a\in A$ and $\lambda \in
\Lambda ^{^{\prime }}$, where $\Lambda ^{^{\prime }}$ is a cofinal subset of 
$\Lambda $, is called an inverse limit automorphism. If $\alpha $ is an
inverse limit automorphism of the pro-$C^{\ast }$-algebra $A$, then $%
X_{\alpha }=\{\xi _{x};x\in A\}$ is a Hilbert $A-A$ pro-$C^{\ast }$-bimodule
with the bimodule structure defined as $\xi _{x}a=\xi _{xa}$, respectively $%
a\xi _{x}=\xi _{\alpha ^{-1}\left( a\right) x}$, and the inner products are
defined as $\left\langle \xi _{x},\xi _{y}\right\rangle _{A}=x^{\ast }y$,
respectively $_{A}\left\langle \xi _{x},\xi _{y}\right\rangle =\alpha \left(
xy^{\ast }\right) $.

Let $\alpha $ and $\beta $ be two inverse limit automorphisms of the pro-$%
C^{\ast }$-algebras $A$, respectively $B$. It is well known that if the pro-$%
C^{\ast }$-dynamical systems $\left( A,\alpha ,\mathbb{Z}\right) $ and $%
\left( B,\beta ,\mathbb{Z}\right) $ are conjugate (that is, there is a pro- $%
C^{\ast }$-morphism $\varphi :A\rightarrow B$ such that $\varphi \circ
\alpha =\beta \circ \varphi $), then the crossed product $A\times _{\alpha }%
\mathbb{Z}$ of $A$ by $\alpha $ and the crossed product $B\times _{\beta }%
\mathbb{Z}$ of $B$ by $\beta $ are isomorphic as pro-$C^{\ast }$-algebras
(see, for example, \cite{W,J4}). Proposition \ref{p1} extends this result
for crossed products by Hilbert pro-$C^{\ast }$-bimodules.

\begin{remark}
Suppose that $\alpha $ and $\beta $ are inverse limit automorphisms of the
pro-$C^{\ast }$-algebras $A$, respectively $B$ and $\varphi :A\rightarrow B$
is a pro-$C^{\ast }$-isomorphism such that $\varphi \circ \alpha =\beta
\circ \varphi $. Then the Hilbert pro-$C^{\ast }$-bimodules $(X_{\alpha },A)$
and $\left( X_{\beta },B\right) $ are isomorphic. Indeed, it is easy to
check that $\left( \varphi _{X_{\alpha }},\varphi _{A}\right) $, where $%
\varphi _{X_{\alpha }}\left( \xi _{x}\right) =\xi _{\varphi \left( x\right)
} $ for all $x\in A$ and $\varphi _{A}=\varphi $, is an isomorphism of
Hilbert pro-$C^{\ast }$-bimodules.
\end{remark}

To prove the following proposition we will use the following lemma.

\begin{lemma}
\label{Conjugate} \cite[Lemma 4.2]{JZ2} Let {$\left( \Phi ,\varphi \right) $
be a }morphism of Hilbert pro-$C^{\ast }$-bimodules from $\left( X,A\right) $
to $\left( Y,B\right) $. If $\Gamma $ and $\Gamma ^{\prime }$ have the same
index set, $\Lambda $, and $\varphi =\lim\limits_{\leftarrow \lambda
}\varphi _{\lambda }$, then $\Phi =\lim\limits_{\leftarrow \lambda }\Phi
_{\lambda },$ for each $\lambda \in \Lambda ,\left( \Phi _{\lambda },\varphi
_{\lambda }\right) $ is a morphism of Hilbert $C^{\ast }$-bimodules, $\left(
\Phi _{\lambda }\times \varphi _{\lambda }\right) _{\lambda }$ is an inverse
system of $C^{\ast }$-morphisms and $\Phi \times \varphi
=\lim\limits_{\leftarrow \lambda }\Phi _{\lambda }\times \varphi _{\lambda }$%
. Moreover, if {$\left( \Phi ,\varphi \right) $ is an isomorphism of }%
Hilbert pro-$C^{\ast }$-bimodules and $\varphi _{\lambda },\lambda \in
\Lambda $ are $C^{\ast }$-isomorphisms, then $\left( \Phi _{\lambda
},\varphi _{\lambda }\right) ,$ $\lambda \in \Lambda $ are isomorphisms of
Hilbert $C^{\ast }$-bimodules.
\end{lemma}

\begin{proposition}
\label{p1}Let $\left( X,A\right) $ and $\left( Y,B\right) $ be \textit{%
Hilbert pro-}$C^{\ast }$\textit{-bimodules. }If the Hilbert pro-$C^{\ast }$%
-bimodules $\left( X,A\right) $ and $\left( Y,B\right) $ are isomorphic,
then the pro-$C^{\ast }$-algebras $A\times _{X}\mathbb{Z}$ to $B\times _{Y}%
\mathbb{Z}$ are isomorphic.
\end{proposition}

\begin{proof}
Let $\left( \Phi ,\varphi \right) :\left( X,A\right) \rightarrow \left(
Y,B\right) $ be an isomorphism of Hilbert pro-$C^{\ast }$-bimodules. Since $%
\varphi $ is a pro-$C^{\ast }$-isomorphism, we can suppose that $\Gamma $
and $\Gamma ^{\prime }$ have the same index set, $\Lambda $, and $\varphi
=\lim\limits_{\leftarrow \lambda }\varphi _{\lambda }$, where $\varphi
_{\lambda }:A_{\lambda }\rightarrow B_{\lambda },$ $\lambda \in \Lambda $
are $C^{\ast }$-isomorphisms \cite[Lemma 3.7]{JZ}. Then, by Lemma \ref%
{Conjugate}, $\Phi =\lim\limits_{\leftarrow \lambda }\Phi _{\lambda }$ and $%
\left( \Phi _{\lambda },\varphi _{\lambda }\right) $ is an isomorphism of
Hilbert $C^{\ast }$-bimodules for all $\lambda \in \Lambda $. By \cite[%
Remark 2.2]{A}, for each $\lambda \in \Lambda ,$ $\Phi _{\lambda }\times
\varphi _{\lambda }:$ $A_{\lambda }\times _{X_{\lambda }}\mathbb{%
Z\rightarrow }B_{\lambda }\times _{Y_{\lambda }}\mathbb{Z}$ is a $C^{\ast }$%
-isomorphism. Moreover, $\left( \Phi _{\lambda }\times \varphi _{\lambda
}\right) _{\lambda }$ is an inverse system of $C^{\ast }$-isomorphisms, and
then $\Phi \times \varphi =\lim\limits_{\leftarrow \lambda }$ $\Phi
_{\lambda }\times \varphi _{\lambda }$ is a pro-$C^{\ast }$-isomorphism.
Therefore, the pro-$C^{\ast }$-algebras $A\times _{X}\mathbb{Z}$ to $B\times
_{Y}\mathbb{Z}$ are isomorphic.
\end{proof}

By \cite{J4}, if $\alpha $ is an inverse limit automorphism of a pro-$%
C^{\ast }$-algebra $A$ and $B$ is another pro-$C^{\ast }$-algebra, then the
minimal tensor product $\left( A\times _{\alpha }\mathbb{Z}\right) \otimes
_{\min }B$ of the crossed product $A\times _{\alpha }\mathbb{Z}$ of $A$ by $%
\alpha $ and $B$ and the crossed product $\left( A\otimes _{\min }B\right)
\times _{\alpha \otimes _{\min }\text{id}}\mathbb{Z}$ of the minimal tensor
product $A\otimes _{\min }B$ of $A$ and $B$ by the automorphism $\alpha
\otimes _{\min }$id are isomorphic as pro-$C^{\ast }$-algebras. Also the
maximal tensor product $\left( A\times _{\alpha }\mathbb{Z}\right) \otimes
_{\max }B$ of the crossed product $A\times _{\alpha }\mathbb{Z}$ of $A$ by $%
\alpha $ and $B$ and the crossed product $\left( A\otimes _{\max }B\right)
\times _{\alpha \otimes _{\max }\text{id}}\mathbb{Z}$ of the maximal tensor
product $A\otimes _{\max }B$ of $A$ and $B$ by the automorphism $\alpha
\otimes _{\max }$id are isomorphic as pro-$C^{\ast }$-algebras. We refer the
reader to \cite{F} for further information about tensor products of pro-$%
C^{\ast }$-algebras. The following theorems extend these results for crossed
products by Hilbert (pro-) $C^{\ast }$-bimodules.

Let $\left( X,A\right) $ and $\left( Y,B\right) $ be two Hilbert pro-$%
C^{\ast }$-bimodules. Then $X\otimes _{\text{alg }}Y$ has a natural
structure of pre-Hilbert $A\otimes _{\text{alg }}B$-$A\otimes _{\text{alg }%
}B $ bimodule with the bimodule structure given by 
\begin{equation*}
\left( a\otimes b\right) \left( x\otimes y\right) =ax\otimes by\text{ and }%
\left( x\otimes y\right) \left( a\otimes b\right) =xa\otimes yb
\end{equation*}%
for all $a\in A,b\in B,$ $x\in X$ and $y\in Y$, and the inner products given
by 
\begin{equation*}
_{A\otimes _{\text{alg }}B}\left\langle x_{1}\otimes y_{1},x_{2}\otimes
y_{2}\right\rangle =\text{ }_{A}\left\langle x_{1},x_{2}\right\rangle
\otimes \text{ }_{B}\left\langle y_{1},y_{2}\right\rangle \text{ }
\end{equation*}%
and%
\begin{equation*}
\left\langle x_{1}\otimes y_{1},x_{2}\otimes y_{2}\right\rangle _{A\otimes _{%
\text{alg }}B}=\text{ }\left\langle x_{1},x_{2}\right\rangle _{A}\otimes 
\text{ }\left\langle y_{1},y_{2}\right\rangle _{B}
\end{equation*}%
for all $x_{1},x_{2}\in X$ and $y_{1},y_{2}\in Y$. Clearly, 
\begin{equation*}
_{A\otimes _{\text{alg }}B}\left\langle x_{1}\otimes y_{1},x_{2}\otimes
y_{2}\right\rangle \left( x_{3}\otimes y_{3}\right) =\left( x_{1}\otimes
y_{1}\right) \left\langle x_{2}\otimes y_{2},x_{3}\otimes y_{3}\right\rangle
_{A\otimes _{\text{alg }}B}
\end{equation*}%
for all $x_{1},x_{2},x_{3}\in X$ and for all $y_{1},y_{2},y_{3}\in Y$. It is
easy to check that 
\begin{equation*}
p_{\left( \lambda ,\delta \right) }^{A\otimes _{\min }B}\left( x\otimes
y\right) =\text{ }^{A\otimes _{\min }B}p_{\left( \lambda ,\delta \right)
}\left( x\otimes y\right)
\end{equation*}%
and 
\begin{equation*}
p_{\left( \lambda ,\delta \right) }^{A\otimes _{\max }B}\left( x\otimes
y\right) =\text{ }^{A\otimes _{\max }B}p_{\left( \lambda ,\delta \right)
}\left( x\otimes y\right)
\end{equation*}%
for all $x\in X,$ $y\in Y\ $and $\left( \lambda ,\delta \right) $ $\in
\Lambda \times \Delta $. Then the completion of $X\otimes _{\text{alg }}Y$
with respect to the topology given by the family of seminorms $\{p_{\left(
\lambda ,\delta \right) }^{A\otimes _{\min }B}\}_{\left( \lambda ,\delta
\right) \in \Lambda \times \Delta }$ is a Hilbert $A\otimes _{\min }B$-$%
A\otimes _{\min }B$ pro-$C^{\ast }$-bimodule, denoted by $X\otimes _{\min
}Y\ $(see, \cite{J1}), and it is called the minimal tensor product of $%
\left( X,A\right) $ and $\left( Y,B\right) $. The completion of $X\otimes _{%
\text{alg }}Y$ with respect to the topology given by the family of seminorms 
$\{p_{\left( \lambda ,\delta \right) }^{A\otimes _{\max }B}\}_{\left(
\lambda ,\delta \right) \in \Lambda \times \Delta }$ is a Hilbert $A\otimes
_{\max }B$-$A\otimes _{\max }B$ pro-$C^{\ast }$-bimodule, denoted by $%
X\otimes _{\max }Y$, and it is called the maximal tensor product of $\left(
X,A\right) $ and $\left( Y,B\right) $. It is easy to check that for each $%
\left( \lambda ,\delta \right) $ $\in \Lambda \times \Delta ,$ the Hilbert $%
C^{\ast }$-bimodules $\left( X\otimes _{\min }Y\right) _{\left( \lambda
,\delta \right) }$ and $X_{\lambda }\otimes _{\min }Y_{\delta }$ are
isomarphic, as well as, the Hilbert $C^{\ast }$-bimodules $\left( X\otimes
_{\max }Y\right) _{\left( \lambda ,\delta \right) }$ and $X_{\lambda
}\otimes _{\max }Y_{\delta }$.

Suppose $\left( X_{i},A_{i}\right) \ $and $\left( Y_{i},B_{i}\right) ,i=1,2$
are four Hilbert pro-$C^{\ast }$-bimodules, and $\left( \varphi
_{X_{i}},\varphi _{A_{i}}\right) :\left( X_{i},A_{i}\right) \rightarrow $ $%
\left( Y_{i},B_{i}\right) ,i=1,2$ are two morphisms of Hilbert pro-$C^{\ast
} $-bimodules. It is not difficult to check that the pair $\left( \varphi
_{X_{1}}\otimes \varphi _{X_{2}},\varphi _{A_{1}}\otimes \varphi
_{A_{2}}\right) $ of linear maps, where $\varphi _{X_{1}}\otimes \varphi
_{X_{2}}:X_{1}\otimes _{\text{alg}}X_{2}\rightarrow Y_{1}\otimes _{\text{alg}%
}Y_{2}$ is given by $\varphi _{X_{1}}\otimes \varphi _{X_{2}}\left(
x_{1}\otimes x_{2}\right) =\varphi _{X_{1}}\left( x_{1}\right) \otimes
\varphi _{X_{2}}\left( x_{2}\right) ,$ $x_{i}\in X_{i},i=1,2$ and $\varphi
_{A_{1}}\otimes \varphi _{A_{2}}:A_{1}\otimes _{\text{alg}}A_{2}\rightarrow
B_{1}\otimes _{\text{alg}}B_{2}$ is given by $\varphi _{A_{1}}\otimes
\varphi _{A_{2}}\left( a_{1}\otimes a_{2}\right) =\varphi _{A_{1}}\left(
a_{1}\right) \otimes \varphi _{A_{2}}\left( a_{2}\right) $, $a_{i}\in
A_{i},i=1,2$, extends to a morphism of Hilbert pro-$C^{\ast }$-bimodules $%
\left( \varphi _{X_{1}}\otimes _{\min }\varphi _{X_{2}},\varphi
_{A_{1}}\otimes _{\min }\varphi _{A_{2}}\right) $ from $\left( X_{1}\otimes
_{\min }X_{2},A_{1}\otimes _{\min }A_{2}\right) $ to $\left( Y_{1}\otimes
_{\min }Y_{2},B_{1}\otimes _{\min }B_{2}\right) $ and a morphism of Hilbert
pro-$C^{\ast }$-bimodules $\left( \varphi _{X_{1}}\otimes _{\max }\varphi
_{X_{2}},\varphi _{A_{1}}\otimes _{\max }\varphi _{A_{2}}\right) $ from $%
\left( X_{1}\otimes _{\max }X_{2},A_{1}\otimes _{\max }A_{2}\right) $ to $%
\left( Y_{1}\otimes _{\max }Y_{2},B_{1}\otimes _{\max }B_{2}\right) $.

We remark that if $\alpha \ $is an inverse limit automorphism of $A$ and $B$
is a pro-$C^{\ast }$-algebra, then the Hilbert pro-$C^{\ast }$-bimodules $%
X_{\alpha \otimes _{\min }\text{id}}$ and $X\otimes _{\min }B$ are
isomorphic, as well as, the Hilbert pro-$C^{\ast }$-bimodules $X_{\alpha
\otimes _{\max }\text{id}}$ and $X\otimes _{\max }B$. The isomorphism is
given by $\left( \varphi _{X_{\alpha \otimes _{\min }\text{id}}},\varphi
_{A\otimes _{\min }B}\right) $, where $\varphi _{X_{\alpha \otimes _{\min }%
\text{id}}}\left( \xi _{x\otimes y}\right) =\xi _{\alpha \left( x\right)
}\otimes \xi _{y};x\in A,y\in B$ and, $\varphi _{A\otimes _{\min }B}\left(
a\otimes b\right) =\alpha \left( a\right) \otimes b;a\in A,b\in B$,
respectively $\left( \varphi _{X_{\alpha \otimes _{\max }\text{id}}},\varphi
_{A\otimes _{\max }B}\right) $, where $\varphi _{X_{\alpha \otimes _{\max }%
\text{id}}}\left( \xi _{x\otimes b}\right) =\xi _{\alpha \left( x\right)
}\otimes \xi _{y};x\in A,y\in B$ and $\varphi _{A\otimes _{\max }B}\left(
a\otimes b\right) =\alpha \left( a\right) \otimes b;a\in A,b\in B$.\bigskip

In order to prove Theorem \ref{tensor} below we are going to use the notion
of dual action on crossed products by Hilbert bimodules \cite[Section 3]{AEE}%
. Let $\left( X,A\right) $ be a Hilbert $C^{\ast }$-bimodule. Then there is
an action $\delta ^{X}$ of $\mathbb{T}$ on $A\times _{X}\mathbb{Z}$, called
the dual action, which is identity on $i_{A}(A)$ and $\delta _{t}^{X}\left(
i_{X}\left( x\right) \right) =ti_{X}\left( x\right) $ for all $t\in \mathbb{T%
}$ and for all $x\in X$.

\begin{theorem}
\label{tensor}Let $\left( X,A\right) $ be a Hilbert pro-$C^{\ast }$-bimodule
and let $B$ be a pro-$C^{\ast }$-algebra. Then the pro-$C^{\ast }$-algebras $%
\left( A\otimes _{\min }B\right) \times _{X\otimes _{\min }B}\mathbb{Z}$ and 
$\left( A\times _{X}\mathbb{Z}\right) \otimes _{\min }$ $B$ are isomorphic.
\end{theorem}

\begin{proof}
We divide the proof into two parts.

\textit{Step 1.} We suppose that $A$ and $B$ are $C^{\ast }$-algebras.

It is easy to check that $\left( i_{X}\otimes _{\min }\text{id}%
_{B},i_{A}\otimes _{\min }\text{id}_{B}\right) $ is a covariant
representation of $\left( X\otimes _{\min }B,A\otimes _{\min }B\right) $ on $%
\left( A\times _{X}\mathbb{Z}\right) $ $\otimes _{\min }$ $B$. By the
universal property of the crossed product by Hilbert $C^{\ast }$-bimodules,
there is a unique $C^{\ast }$-morphism $\Phi :\left( A\otimes _{\min
}B\right) \times _{X\otimes _{\min }B}\mathbb{Z}$ $\rightarrow $ $\left(
A\times _{X}\mathbb{Z}\right) \otimes _{\min }$ $B$ such that $\Phi \circ
i_{X\otimes _{\min }B}=i_{X}\otimes _{\min }$id$_{B}$ and $\Phi \circ
i_{A\otimes _{\min }B}=i_{A}\otimes _{\min }$id$_{B}$. Since $\Phi $ is a $%
C^{\ast }$-morphism, and since $\left( A\times _{X}\mathbb{Z}\right) \otimes
_{\min }$ $B$ is generated by $i_{X}(X)\otimes _{\text{alg}}B$ and $%
i_{A}(A)\otimes _{\text{alg}}B$, $\Phi $ is surjective.

On the other hand, $\Phi |_{i_{A\otimes _{\min }B}\left( A\otimes _{\min
}B\right) }$ is injective, since $\Phi |_{i_{A\otimes _{\min }B}\left(
A\otimes _{\min }B\right) }=i_{A}\otimes $id$_{B}$ and $i_{A}$ is injective.
If $\delta ^{X\otimes _{\min }B}$ and $\delta ^{X}$ are the dual actions of $%
\mathbb{T}$ on $\left( A\otimes _{\min }B\right) $ $\times _{X\otimes _{\min
}B}\mathbb{Z},$ respectively $A\times _{X}\mathbb{Z}$ (see \cite[Section 3]%
{AEE}), and $\delta ^{X}\otimes $id is an action of $\mathbb{T}$ on $\left(
A\times _{X}\mathbb{Z}\right) \otimes _{\min }B$ given by $\left( \delta
^{X}\otimes \text{id}\right) _{t}=\delta _{t}^{X}\otimes $id$_{B}$, then $%
\Phi $ is $\left( \delta ^{X\otimes _{\min }B},\delta ^{X}\otimes \text{id}%
\right) $-covariant, that is $\Phi \circ \delta _{t}^{X\otimes _{\min
}B}=\left( \delta ^{X}\otimes \text{id}\right) _{t}\circ \Phi $ for all $%
t\in \mathbb{T}$, since%
\begin{eqnarray*}
&&\left( \Phi \circ \delta _{t}^{X\otimes _{\min }B}\right) \left(
i_{X\otimes _{\min }B}\left( x\otimes b\right) \right) =\Phi \left(
ti_{X\otimes _{\min }B}\left( x\otimes b\right) \right) \\
&=&t\Phi \left( i_{X\otimes _{\min }B}\left( x\otimes b\right) \right)
=ti_{X}\left( x\right) \otimes b=\left( \delta ^{X}\otimes \text{id}\right)
_{t}\left( i_{X}\left( x\right) \otimes b\right) \\
&=&\left( \delta ^{X}\otimes \text{id}\right) _{t}\left( \Phi \left(
i_{X\otimes _{\min }B}\left( x\otimes b\right) \right) \right) =\left(
\left( \delta ^{X}\otimes \text{id}\right) _{t}\circ \Phi \right) \left(
i_{X\otimes _{\min }B}\left( x\otimes b\right) \right)
\end{eqnarray*}%
for all $x\in X$, $b\in B$, $t\in \mathbb{T}$ and%
\begin{eqnarray*}
&&\left( \Phi \circ \delta _{t}^{X\otimes _{\min }B}\right) \left(
i_{A\otimes _{\min }B}\left( a\otimes b\right) \right) =\Phi \left(
i_{A\otimes _{\min }B}\left( a\otimes b\right) \right) =i_{A}\left( a\right)
\otimes b \\
&=&\left( \delta ^{X}\otimes \text{id}\right) _{t}\left( i_{A}\left(
a\right) \otimes b\right) =\left( \delta ^{X}\otimes \text{id}\right)
_{t}\left( \Phi \left( i_{A\otimes _{\min }B}\left( a\otimes b\right)
\right) \right) \\
&=&\left( \left( \delta ^{X}\otimes \text{id}\right) _{t}\circ \Phi \right)
\left( i_{A\otimes _{\min }B}\left( a\otimes b\right) \right)
\end{eqnarray*}%
for all $a\in A$, $b\in B$ and $t\in \mathbb{T}$. Then, by \cite[Proposition
2.9]{E}, $\Phi $ is injective, and so it is a $C^{\ast }$-isomorphism.

\textit{Step 2.} We suppose that $A$ and $B$ are pro-$C^{\ast }$-algebras.

By \cite[Proposition 3.8]{JZ}, we can suppose that $A\times _{X}\mathbb{Z=}%
\lim\limits_{\leftarrow \lambda }A_{\lambda }\times _{X_{\lambda }}\mathbb{Z}
$ with $\left( A\times _{X}\mathbb{Z}\right) _{\lambda }=A_{\lambda }\times
_{X_{\lambda }}\mathbb{Z}$ for each $\lambda \in \Lambda ,$ and $\left(
A\otimes _{\min }B\right) \times _{X\otimes _{\min }B}\mathbb{Z=}%
\lim\limits_{\leftarrow \left( \lambda ,\delta \right) }$ $\left( A_{\lambda
}\otimes _{\min }B_{\delta }\right) \times _{X_{\lambda }\otimes _{\min
}B_{\delta }}\mathbb{Z\ \ }$ with\ $\left( \left( A\otimes _{\min }B\right)
\times _{X\otimes _{\min }B}\mathbb{Z}\right) _{\left( \lambda ,\delta
\right) }=\left( A_{\lambda }\otimes _{\min }B_{\delta }\right) $ $\ \ \
\times _{X_{\lambda }\otimes _{\min }B_{\delta }}\mathbb{Z}$ for each $%
(\lambda ,\delta )\in \Lambda \times \Delta $.

Let $(\lambda ,\delta )\in \Lambda \times \Delta $. According to\textit{\
Step 1}, there is a $C^{\ast }$-isomorphism $\Phi _{(\lambda ,\delta
)}:\left( A_{\lambda }\otimes _{\min }B_{\delta }\right) \times _{X_{\lambda
}\otimes _{\min }B_{\delta }}\mathbb{Z}$ $\rightarrow $ $\left( A_{\lambda
}\times _{X_{\lambda }}\mathbb{Z}\right) \otimes _{\min }$ $B_{\delta }$
such that $\Phi _{(\lambda ,\delta )}\circ i_{X_{\lambda }\otimes _{\min
}B_{\delta }}=i_{X_{\lambda }}\otimes $id$_{B_{\delta }}$ and $\Phi
_{(\lambda ,\delta )}\circ i_{A_{\lambda }\otimes _{\min }B_{\delta
}}=i_{A_{\lambda }}\otimes $id$_{B_{\delta }}$. $\ $From 
\begin{eqnarray*}
&&\pi _{\left( \lambda _{1},\delta _{1}\right) \left( \lambda _{2},\delta
_{2}\right) }^{\left( A\times _{X}\mathbb{Z}\right) \otimes _{\min }B}\circ
\Phi _{\left( \lambda _{1},\delta _{1}\right) }\left( i_{X_{\lambda
_{1}}\otimes _{\min }B_{\delta _{1}}}\left( z\right) \right) =\pi _{\left(
\lambda _{1},\delta _{1}\right) \left( \lambda _{2},\delta _{2}\right)
}^{\left( A\times _{X}\mathbb{Z}\right) \otimes _{\min }B}\left(
i_{X_{\lambda _{1}}}\otimes \text{id}_{B_{\delta _{1}}}\left( z\right)
\right) \\
&=&\pi _{\lambda _{1}\lambda _{2}}^{A\times _{X}\mathbb{Z}}\otimes \pi
_{\delta _{1}\delta _{2}}^{B}\left( i_{X_{\lambda _{1}}}\otimes \text{id}%
_{B_{\delta _{1}}}\left( z\right) \right) =i_{X_{\lambda _{2}}}\otimes \text{%
id}_{B_{\delta _{2}}}\left( \sigma _{\lambda _{1}\lambda _{2}}^{X}\otimes
\pi _{\delta _{1}\delta _{2}}^{B}\left( z\right) \right) \\
&=&\Phi _{\left( \lambda _{2},\delta _{2}\right) }\circ i_{X_{\lambda
_{2}}\otimes _{\min }B_{\delta _{2}}}\left( \sigma _{\lambda _{1}\lambda
_{2}}^{X}\otimes \pi _{\delta _{1}\delta _{2}}^{B}\left( z\right) \right) \\
&=&\Phi _{\left( \lambda _{2},\delta _{2}\right) }\circ \pi _{\left( \lambda
_{1},\delta _{1}\right) \left( \lambda _{2},\delta _{2}\right) }^{\left(
A\times _{X}\mathbb{Z}\right) \otimes _{\min }B}\left( i_{X_{\lambda
_{1}}\otimes _{\min }B_{\delta _{1}}}\left( z\right) \right)
\end{eqnarray*}%
for all $z\in X_{\lambda _{1}}\otimes _{\min }B_{\delta _{1}}$ and 
\begin{eqnarray*}
&&\pi _{\left( \lambda _{1},\delta _{1}\right) \left( \lambda _{2},\delta
_{2}\right) }^{\left( A\times _{X}\mathbb{Z}\right) \otimes _{\min }B}\circ
\Phi _{\left( \lambda _{1},\delta _{1}\right) }\left( i_{A_{\lambda
_{1}}\otimes _{\min }B_{\delta _{1}}}\left( a\otimes b\right) \right) \\
&=&\pi _{\left( \lambda _{1},\delta _{1}\right) \left( \lambda _{2},\delta
_{2}\right) }^{\left( A\times _{X}\mathbb{Z}\right) \otimes _{\min }B}\left(
i_{A_{\lambda _{1}}}\left( a\right) \otimes \text{id}_{B_{\delta
_{1}}}\left( b\right) \right) =\pi _{\lambda _{1}\lambda _{2}}^{A\times _{X}%
\mathbb{Z}}\otimes \pi _{\delta _{1}\delta _{2}}^{B}\left( i_{A_{\lambda
_{1}}}\left( a\right) \otimes b\right) \\
&=&i_{A_{\lambda _{2}}}\otimes \text{id}_{B_{\delta _{2}}}\left( \pi
_{\lambda _{1}\lambda _{2}}^{A}\left( a\right) \otimes \pi _{\delta
_{1}\delta _{2}}^{B}\left( b\right) \right) \\
&=&\Phi _{\left( \lambda _{2},\delta _{2}\right) }\circ i_{A_{\lambda
_{2}}\otimes _{\min }B_{\delta _{2}}}\left( \pi _{\lambda _{1}\lambda
_{2}}^{A}\left( a\right) \otimes \pi _{\delta _{1}\delta _{2}}^{B}\left(
b\right) \right) \\
&=&\Phi _{\left( \lambda _{2},\delta _{2}\right) }\circ \pi _{\left( \lambda
_{1},\delta _{1}\right) \left( \lambda _{2},\delta _{2}\right) }^{\left(
A\times _{X}\mathbb{Z}\right) \otimes _{\min }B}\left( i_{A_{\lambda
_{1}}\otimes _{\min }B_{\delta _{1}}}\left( a\otimes b\right) \right)
\end{eqnarray*}%
for all $a\in A_{\lambda _{1}},b\in \ B_{\delta _{1}}$ and for all $\left(
\lambda _{1},\delta _{1}\right) ,\left( \lambda _{2},\delta _{2}\right) \in
\Lambda \times \Delta $ with $\left( \lambda _{1},\delta _{1}\right) \geq
\left( \lambda _{2},\delta _{2}\right) $, and taking into account that $%
\left( A_{\lambda }\otimes _{\min }B_{\delta }\right) \times _{X_{\lambda
}\otimes _{\min }B_{\delta }}\mathbb{Z}$ is generated by $i_{X_{\lambda
}\otimes _{\min }B_{\delta }}\left( z\right) $ and $i_{A_{\lambda }\otimes
_{\min }B_{\delta }}\left( a\otimes b\right) $ with $z\in X_{\lambda
}\otimes _{\min }B_{\delta }$ and $a\in A_{\lambda },b\in \ B_{\delta }$, we
deduce that $\left( \Phi _{\left( \lambda ,\delta \right) }\right) _{\left(
\lambda ,\delta \right) }$ is an inverse system of $C^{\ast }$-isomorphisms.
Then $\Phi =\lim\limits_{\leftarrow \left( \lambda ,\delta \right) }\Phi
_{\left( \lambda ,\delta \right) }$ is a pro-$C^{\ast }$-isomorphism from $%
\left( A\otimes _{\min }B\right) \times _{X\otimes _{\min }B}\mathbb{Z}$ $\ $%
to $\left( A\times _{X}\mathbb{Z}\right) \otimes $ $_{\min }B$.
\end{proof}

Let $X$ and $Y$ be Hilbert pro-$C^{\ast }$-modules over $A$. A morphism $%
T:X\rightarrow Y$ of right modules is \textit{adjointable} if there is
another morphism of modules $T^{\ast }:Y\rightarrow X\ $such that $%
\left\langle Tx_{1},x_{2}\right\rangle _{A}=\left\langle x_{1},T^{\ast
}x_{2}\right\rangle _{A}$ for all $x_{1},x_{2}\in X$. The vector space $%
L_{A}(X,Y)\ $of all adjointable module morphisms from $X$ to $Y$ has a
structure of locally convex space under the topology given by the family of
seminorms $\{p_{\lambda ,L_{A}(X,Y)}\}_{\lambda \in \Lambda }$, where $%
p_{\lambda ,L_{A}(X,Y)}\left( T\right) =\sup \{p_{\lambda
}^{A}(Tx);p_{\lambda }^{A}\left( x\right) \leq 1\}$.

For $Y=X,$ $L_{A}(X)=L_{A}(X,X)$ is a pro-$C^{\ast }$-algebra with $\left(
L_{A}(X)\right) _{\lambda }=L_{A_{\lambda }}(X_{\lambda })$ for each $%
\lambda \in \Lambda $,

The vector space $L_{A}(X,Y)$ of all adjointable module map from $X$ to $Y$
has a natural structure of Hilbert $L_{A}(Y)-L_{A}(X)$ pro-$C^{\ast }$%
-bimodule with the bimodule structure given by 
\begin{equation*}
S\cdot T=S\circ T\text{ and }T\cdot R=T\circ R\ 
\end{equation*}%
for\ all\ $T\in L_{A}(X,Y),S\in L_{A}(Y)\ $and\ $R\in L_{A}(X)\ $and the
inner products given by 
\begin{equation*}
_{L_{A}(Y)}\left\langle T_{1},T_{2}\right\rangle =T_{1}\circ T_{2}^{\ast }%
\text{ and }\left\langle T_{1},T_{2}\right\rangle _{L_{A}(X)}=T_{1}^{\ast
}\circ T_{2}
\end{equation*}%
for all $T_{1},T_{2}\in L_{A}(X,Y)$.

Suppose that $X$ is a full Hilbert $A-A$ pro-$C^{\ast }$-bimodule. Then $%
L_{A}(X)$ can be identified with $M(A)$, the multiplier algebra of $A$, and $%
L_{A}(A,X),$ has a structure of Hilbert $M(A)-M(A)$ pro-$C^{\ast }$%
-bimodule. $L_{A}(A,X)$ is denoted by $M(X)\ $and it is called the\textit{\
multiplier bimodule} of $X\ $(see \cite{JMZ}). Moreover, $X$ can be regarded
as a pro-$C^{\ast }$-sub-bimodule of $M(X)$, and $M(X)_{\lambda
}=M(X_{\lambda })$ for all $\lambda \in \Lambda $.

A morphism of Hilbert pro-$C^{\ast }$-bimodules\textit{\ }$\left( \Phi
,\varphi \right) :\left( X,A\right) \rightarrow \left( M(Y),M(B)\right) $ is 
\textit{nondegenerate if } the pro-$C^{\ast }$-subalgebra generated by $%
\{\varphi \left( a\right) b;a\in A,b\in B\}$ is dense in $B$ and vector
space generated by $\{\Phi \left( x\right) b;x\in X,b\in B\}\ $is dense in $%
Y $. A nondegenerate morphism of Hilbert pro-$C^{\ast }$-bimodules\textit{\ }%
$\left( \Phi ,\varphi \right) :\left( X,A\right) \rightarrow \left(
M(Y),M(B)\right) $ extends to a unique morphism of Hilbert pro-$C^{\ast }$%
-bimodules $\left( \overline{\Phi },\overline{\varphi }\right) :\left(
M(X),M(A)\right) \rightarrow \left( M(Y),M(B)\right) $. The canonical
representation $\left( i_{X},i_{A}\right) $ of $\left( X,A\right) $ on $%
A\times _{X}\mathbb{Z}$ is nondegenerate.

\begin{lemma}
\label{aj}Let $\left( X,A\right) $ be a full Hilbert pro-$C^{\ast }$%
-bimodule and let $B$ be a pro-$C^{\ast }$-algebra. Then there is a
nondegenerate morphism of Hilbert pro-$C^{\ast }$-bimodules\ \ \ $\left(
\rho _{X\otimes _{\max }B}^{X},\rho _{A\otimes _{\max }B}^{A}\right) :\left(
X,A\right) \rightarrow \left( M(X\otimes _{\max }B),M\left( A\otimes _{\max
}B\right) \right) \ \ \ \ \ \ $ such \ \ that\ \ \ $\ \ \ \rho _{X\otimes
_{\max }B}^{X}\left( x\right) \left( a\otimes b\right) =xa\otimes b\ $for
all $x\in X,a\in A$ and $b\in B$ and $\rho _{A\otimes _{\max }B}^{A}\left(
c\right) \left( a\otimes b\right) =ca\otimes b\ $for all $a,c\in A$ and $%
b\in B$.
\end{lemma}

\begin{proof}
We divide the proof into two parts.

\textit{Step 1.} Suppose that $A$ and $B$ are $C^{\ast }$-algebras.

Let $x\in X$. Consider the linear maps $\rho _{X\otimes _{\max }B}^{X}\left(
x\right) :A\otimes _{\text{alg}}B\rightarrow X\otimes _{\text{alg}}B$ given
by $\rho _{X\otimes _{\max }B}^{X}\left( x\right) \left( a\otimes b\right)
=xa\otimes b,$ and $\rho _{X\otimes _{\max }B}^{X}\left( x\right) ^{\ast
}:X\otimes _{\text{alg}}B\rightarrow A\otimes _{\text{alg}}B$ given by $\rho
_{X\otimes _{\max }B}^{X}\left( x\right) ^{\ast }\left( y\otimes b\right)
=\left\langle x,y\right\rangle _{A}\otimes b$. From

\begin{eqnarray*}
&&\left\Vert \rho _{X\otimes _{\max }B}^{X}\left( x\right) \left(
\tsum\limits_{i=1}^{n}a_{i}\otimes b_{i}\right) \right\Vert _{X\otimes
_{\max }B}^{2}=\left\Vert \tsum\limits_{i,j=1}^{n}a_{i}^{\ast }\left\langle
x,x\right\rangle _{A}a_{j}\otimes b_{i}^{\ast }b_{j}\right\Vert _{A\otimes
_{\max }B} \\
&\leq &\left\Vert x\right\Vert _{X}^{2}\left\Vert
\tsum\limits_{i,j=1}^{n}a_{i}^{\ast }a_{j}\otimes b_{i}^{\ast
}b_{j}\right\Vert _{A\otimes _{\max }B}=\left\Vert x\right\Vert
_{X}^{2}\left\Vert \tsum\limits_{i,j=1}^{n}a_{i}\otimes b_{i}\right\Vert
_{A\otimes _{\max }B}^{2}
\end{eqnarray*}%
for all $a_{1},...,a_{n}\in A$ and $b_{1},...,b_{n}\in B$, we deduce that $%
\rho _{X\otimes _{\max }B}^{X}\left( x\right) $ extends to a continuous
linear map from $A\otimes _{\max }B$ to $X\otimes _{\max }B$. From 
\begin{eqnarray*}
&&\left\Vert \rho _{X\otimes _{\max }B}^{X}\left( x\right) ^{\ast }\left(
\tsum\limits_{i=1}^{n}x_{i}\otimes b_{i}\right) \right\Vert _{A\otimes
_{\max }B}^{2}=\left\Vert \tsum\limits_{i,j=1}^{n}\left\langle
x_{i},x\right\rangle _{A}\left\langle x,x_{j}\right\rangle _{A}\otimes
b_{i}^{\ast }b_{j}\right\Vert _{A\otimes _{\max }B} \\
&\leq &\left\Vert x\right\Vert _{X}^{2}\left\Vert
\tsum\limits_{i,j=1}^{n}\left\langle x_{i},x_{j}\right\rangle _{A}\otimes
b_{i}^{\ast }b_{j}\right\Vert _{A\otimes _{\max }B}=\left\Vert x\right\Vert
_{X}^{2}\left\Vert \tsum\limits_{i,j=1}^{n}x_{i}\otimes b_{i}\right\Vert
_{X\otimes _{\max }B}^{2}
\end{eqnarray*}%
for all $a_{1},...,a_{n}\in A$ and $x_{1},...,x_{n}\in X$, we deduce that $%
\rho _{X\otimes _{\max }B}^{X}\left( x\right) ^{\ast }$ extends to a
continuous linear map from $X\otimes _{\max }B$ to$A\otimes _{\max }B$.
Since 
\begin{eqnarray*}
&&\left\langle \rho _{X\otimes _{\max }B}^{X}\left( x\right) \left(
\tsum\limits_{i=1}^{n}a_{i}\otimes b_{i}\right)
,\tsum\limits_{j=1}^{m}x_{j}\otimes c_{j}\right\rangle _{A\otimes _{\max }B}
\\
&=&\tsum\limits_{i=1}^{n}\tsum\limits_{j=1}^{m}\left\langle xa_{i}\otimes
b_{i},x_{j}\otimes c_{j}\right\rangle _{A\otimes _{\max
}B}=\tsum\limits_{i=1}^{n}\tsum\limits_{j=1}^{m}a_{i}^{\ast }\left\langle
x,x_{j}\right\rangle _{A}b_{i}^{\ast }c_{j} \\
&=&\left\langle \tsum\limits_{i=1}^{n}a_{i}\otimes
b_{i},\tsum\limits_{j=1}^{m}\left\langle x,x_{j}\right\rangle _{A}\otimes
c_{j}\right\rangle _{A\otimes _{\max }B} \\
&=&\left\langle \tsum\limits_{i=1}^{n}a_{i}\otimes b_{i},\rho _{X\otimes
_{\max }B}^{X}\left( x\right) ^{\ast }\left(
\tsum\limits_{j=1}^{m}x_{j}\otimes c_{j}\right) \right\rangle _{A\otimes
_{\max }B}
\end{eqnarray*}%
for all $a_{1},...,a_{n}\in A,$ $b_{1},...,b_{m},c_{1},...,c_{m}\in B$ and $%
x_{1},...,x_{m}\in X,$ we have $\rho _{X\otimes _{\max }B}^{X}\left(
x\right) $ $\in M(X\otimes _{\max }B)$. Let $\rho _{A\otimes _{\max }B}^{A}$
be the inclusion of $A$ into $M(A\otimes _{\max }B)\ $(see, for example, 
\cite[Theorem B27]{RW}). Then $\rho _{A\otimes _{\max }B}^{A}\left( c\right)
\left( a\otimes b\right) =ca\otimes b$ for all $a,c\in A,$ $b\in B$, and it
is easy to check that $\left( \rho _{X\otimes _{\max }B}^{X},\rho _{A\otimes
_{\max }B}^{A}\right) $ is a morphism of Hilbert $C^{\ast }$-bimodules.
Moreover, $\left( \rho _{X\otimes _{\max }B}^{X},\rho _{A\otimes _{\max
}B}^{A}\right) $ is nondegenerate.

\textit{Step 2.} Suppose that $A$ and $B$ are pro-$C^{\ast }$-algebras. Let $%
\left( \lambda ,\delta \right) \in \Lambda \times \Delta $ and $\left( \rho
_{X_{\lambda }\otimes _{\max }B_{\delta }}^{X_{\lambda }},\rho _{A_{\lambda
}\otimes _{\max }B_{\delta }}^{A_{\lambda }}\right) :\left( X_{\lambda
},A_{\lambda }\right) \rightarrow \left( M(X_{\lambda }\otimes _{\max
}B_{\delta }),M\left( A_{\lambda }\otimes _{\max }B_{\delta }\right) \right) 
$ be the nondegenerate morphism of Hilbert $C^{\ast }$-bimodules constructed
in \textit{Step1.} It is easy to check that $\left( \rho _{X_{\lambda
}\otimes _{\max }B_{\delta }}^{X_{\lambda }}\right) _{\left( \lambda ,\delta
\right) \in \Lambda \times \Delta }$ and $\left( \rho _{A_{\lambda }\otimes
_{\max }B_{\delta }}^{A_{\lambda }}\right) _{\left( \lambda ,\delta \right)
\in \Lambda \times \Delta }$ are inverse systems of linear maps,
respectively $C^{\ast }$-morphisms, and $\left( \lim\limits_{\leftarrow
\left( \lambda ,\delta \right) }\rho _{X_{\lambda }\otimes _{\max }B_{\delta
}}^{X_{\lambda }},\lim\limits_{\leftarrow \left( \lambda ,\delta \right)
}\rho _{A_{\lambda }\otimes _{\max }B_{\delta }}^{A_{\lambda }}\right)
:\left( \lim\limits_{\leftarrow \lambda }X_{\lambda
},\lim\limits_{\leftarrow \lambda }A_{\lambda }\right) \rightarrow \left(
\lim\limits_{\leftarrow \left( \lambda ,\delta \right) }X_{\lambda }\otimes
_{\max }B_{\delta },\lim\limits_{\leftarrow \left( \lambda ,\delta \right)
}A_{\lambda }\otimes _{\max }B_{\delta }\right) $ is a morphism of Hilbert
pro-$C^{\ast }$-bimodules. \ Identifying\ \ \ $\left(
\lim\limits_{\leftarrow \left( \lambda ,\delta \right) }X_{\lambda }\otimes
_{\max }B_{\delta },\lim\limits_{\leftarrow \left( \lambda ,\delta \right)
}A_{\lambda }\otimes _{\max }B_{\delta }\right) $\ \ \ with $\left( X\otimes
_{\max }B,A\otimes _{\max }B\right) \ $and $\left( \lim\limits_{\leftarrow
\lambda }X_{\lambda },\lim\limits_{\leftarrow \lambda }A_{\lambda }\right) $
with $\left( X,A\right) $, we obtain a nondegenerate morphism of Hilbert pro-%
$C^{\ast }$-bimodules $\left( \rho _{X\otimes _{\max }B}^{X},\rho _{A\otimes
_{\max }B}^{A}\right) :\left( X,A\right) \rightarrow (M(X\otimes _{\max }B),$
\ $M\left( A\otimes _{\max }B\right) )$\ \ such that $\rho _{X\otimes _{\max
}B}^{X}\left( x\right) \left( a\otimes b\right) =xa\otimes b\ $for all $x\in
X,a\in A$ and $b\in B$ and $\rho _{A\otimes _{\max }B}^{A}\left( c\right)
\left( a\otimes b\right) =ca\otimes b\ $for all $a,c\in A$ and $b\in B$.
\end{proof}

\begin{theorem}
\label{tensor1}Let $\left( X,A\right) $ be a full Hilbert pro-$C^{\ast }$%
-bimodule and let $B$ be a pro-$C^{\ast }$-algebra. Then the pro-$C^{\ast }$%
-algebras $\left( A\otimes _{\max }B\right) \times _{X\otimes _{\max }B}%
\mathbb{Z}$ and $\left( A\times _{X}\mathbb{Z}\right) \otimes _{\max }$ $B$
are isomorphic.
\end{theorem}

\begin{proof}
\textit{Step 1.} Suppose that $A$\ and $B$\ are $C^{\ast }$-algebras.

Since\textbf{\ }$\left( i_{X}\otimes _{\max }\text{id}_{B}\text{,}%
i_{A}\otimes _{\max }\text{id}_{B}\right) $\textbf{\ }$\ $is a covariant
representation of $\ (X\otimes _{\max }B,$\ $A\otimes _{\max }B)$\ on $%
\left( A\times _{X}\mathbb{Z}\right) \otimes _{\max }B$\textbf{, \ }there \
is a unique $C^{\ast }$-morphism $\Phi :\left( A\otimes _{\max }B\right) $\ $%
\times _{X\otimes _{\max }B}Z\rightarrow $\ $\left( A\times _{X}\mathbb{Z}%
\right) \otimes _{\max }$\ $B$\ such that $\Phi \circ i_{X\otimes _{\max
}B}=i_{X}\otimes _{\max }$id$_{B}$\ and $\Phi \circ i_{A\otimes _{\max
}B}=i_{A}\otimes _{\max }$id$_{B}$. Moreover, $\Phi $\ is surjective.

Let $\rho _{A\otimes _{\max }B}^{A}$ and $\rho _{A\otimes _{\max }B}^{B}$
the canonical inclusions of $A$ and $B$ in $M\left( A\otimes _{\max
}B\right) $ and $\rho _{X\otimes _{\max }B}^{X}:X\rightarrow M\left(
X\otimes _{\max }B\right) $ the map constructed in Lemma \ref{aj}. Then $%
\left( \overline{i_{i_{X\otimes _{\max }B}}}\circ \rho _{X\otimes _{\max
}B}^{X},\overline{i_{A\otimes _{\max }B}}\circ \rho _{A\otimes _{\max
}B}^{A}\right) $ is a covariant representation of $\left( X,A\right) $ on $%
M\left( \left( A\otimes _{\max }B\right) \times _{X\otimes _{\max }B}\mathbb{%
Z}\right) ,$ and by the universal property, there is a unique $C^{\ast \ }$%
-morphism $\Psi _{1}:A\times _{X}\mathbb{Z\rightarrow }M\left( \left(
A\otimes _{\max }B\right) \times _{X\otimes _{\max }B}\mathbb{Z}\right) $
such that $\Psi _{1}\circ i_{X}=\overline{i_{X\otimes _{\max }B}}\circ \rho
_{X\otimes _{\max }B}^{X}$ and $\Psi _{1}\circ i_{A}=\overline{i_{A\otimes
_{\max }B}}\circ \rho _{A\otimes _{\max }B}^{A}$.

Let $\Psi _{2}=\overline{i_{A\otimes _{\max }B}}\circ \rho _{A\otimes _{\max
}B}^{B}:B\rightarrow M\left( \left( A\otimes _{\max }B\right) \times
_{X\otimes _{\max }B}\mathbb{Z}\right) $. Since

\begin{eqnarray*}
&&\Psi _{1}\left( i_{A}\left( a\right) \right) \Psi _{2}\left( b\right) =%
\overline{i_{A\otimes _{\max }B}}\left( \rho _{A\otimes _{\max }B}^{A}\left(
a\right) \rho _{A\otimes _{\max }B}^{B}\left( b\right) \right) =i_{A\otimes
_{\max }B}\left( a\otimes b\right) \\
&=&\overline{i_{A\otimes _{\max }B}}\left( \rho _{A\otimes _{\max
}B}^{B}\left( b\right) \rho _{A\otimes _{\max }B}^{A}\left( a\right) \right)
=\Psi _{2}\left( b\right) \Psi _{1}\left( i_{A}\left( a\right) \right)
\end{eqnarray*}
for all $a\in A,b\in B$ and%
\begin{eqnarray*}
&&\Psi _{1}\left( i_{X}\left( x\right) \right) \Psi _{2}\left( b\right) =%
\overline{i_{X\otimes _{\max }B}}\left( \rho _{X\otimes _{\max }B}^{X}\left(
x\right) \right) \overline{i_{A\otimes _{\max }B}}\left( \rho _{A\otimes
_{\max }B}^{B}\left( b\right) \right) \\
&=&\overline{i_{X\otimes _{\max }B}}\left( \rho _{X\otimes _{\max
}B}^{X}\left( x\right) \rho _{A\otimes _{\max }B}^{B}\left( b\right) \right)
=i_{X\otimes _{\max }B}\left( x\otimes b\right) =\Psi _{2}\left( b\right)
\Psi _{1}\left( i_{X}\left( x\right) \right)
\end{eqnarray*}
for all $x\in X,b\in B$, and since $i_{A}\left( A\right) $ and $i_{X}\left(
X\right) $ generate $A\times _{X}\mathbb{Z}$, there is a unique $C^{\ast }$%
-morphism $\Psi :\left( A\times _{X}\mathbb{Z}\right) \otimes _{\max }B$ $%
\rightarrow M\left( \left( A\otimes _{\max }B\right) \times _{X\otimes
_{\max }B}\mathbb{Z}\right) $ such that $\Psi \left( c\otimes b\right) =\Psi
_{1}\left( c\right) \Psi _{2}\left( b\right) $ for all $c\in A\times _{X}%
\mathbb{Z}$ and $b\in B$. Moreover, $\Psi \left( \left( A\times _{X}\mathbb{Z%
}\right) \otimes _{\max }B\right) $ $\subseteq \left( A\otimes _{\max
}B\right) \times _{X\otimes _{\max }B}\mathbb{Z}.$

From 
\begin{equation*}
\Psi \circ \Phi \left( i_{X\otimes _{\max }B}\left( x\otimes b\right)
\right) =\Psi \left( i_{X}\left( x\right) \otimes b\right) =\Psi _{2}\left(
b\right) \Psi _{1}\left( i_{X}\left( x\right) \right) =i_{X\otimes _{\max
}B}\left( x\otimes b\right)
\end{equation*}%
and 
\begin{equation*}
\Psi \circ \Phi \left( i_{A\otimes _{\max }B}\left( a\otimes b\right)
\right) =\Psi \left( i_{A}\left( a\right) \otimes b\right) =\Psi _{2}\left(
b\right) \Psi _{1}\left( i_{A}\left( a\right) \right) =i_{A\otimes _{\max
}B}\left( a\otimes b\right)
\end{equation*}%
for all $a\in A,b\in B\ $and for all $x\in X$ and taking into account that $%
i_{X\otimes _{\max }B}\left( X\otimes _{\max }B\right) $ and $i_{A\otimes
_{\max }B}\left( A\otimes _{\max }B\right) $ generate $\left( A\otimes
_{\max }B\right) \times _{X\otimes _{\max }B}\mathbb{Z}$, we conclude $\Psi
\circ \Phi =$\ \ \ \ \ id$_{\left( A\otimes _{\max }B\right) \times
_{X\otimes _{\max }B}\mathbb{Z}}$,$\ $where using the fact that $\Phi \ $is
surjective, we deduce that $\Phi $ is a $C^{\ast }$-isomorphism.

\textit{Step 2. }The general case.\ Let $\left( \lambda ,\delta \right) \in
\Lambda \times \Delta $ and let $\Phi _{\left( \lambda ,\delta \right) }$ $:$
$\left( A_{\lambda }\otimes _{\max }B_{\delta }\right) \times _{X_{\lambda
}\otimes _{\max }B_{\delta }}Z\rightarrow $\textbf{\ }$\left( A_{\lambda
}\times _{X_{\lambda }}\mathbb{Z}\right) \otimes _{\max }B_{\delta }$\ be
the $C^{\ast }$-isomorphism constructed in \textit{Step 1. }Following the
proof of Theorem \ref{tensor} \textit{Step 2},\textit{\ } we obtain a pro-$%
C^{\ast }$-isomorphism $\Phi =\lim\limits_{\leftarrow \left( \lambda ,\delta
\right) }\Phi _{\left( \lambda ,\delta \right) }$ from $\left( A\otimes
_{\max }B\right) \times _{X\otimes _{\max }B}\mathbb{Z}$ to $\left( A\times
_{X}\mathbb{Z}\right) \otimes _{\max }$ $B$.
\end{proof}

\section{An application}

A pro-$C^{\ast }$-algebra $A$ is nuclear if for any pro-$C^{\ast }$-algebra $%
B,$ $A\otimes _{\min }B=A\otimes _{\max }B\ $\cite[p. 126]{P}.

\begin{remark}
If $X$ is a Hilbert $C^{\ast }$-bimodule over a nuclear pro-$C^{\ast }$%
-algebra $A$, then, clearly, $X\otimes _{\min }B=X\otimes _{\max }B$ for any
pro-$C^{\ast }$-algebra $B$, and $\left( A\otimes _{\min }B\right) \times
_{X\otimes _{\min }B}\mathbb{Z}$ $=$ $\left( A\otimes _{\max }B\right)
\times _{X\otimes _{\max }B}\mathbb{Z}.$
\end{remark}

Suppose that the topology on $A$ is given by the family of $C^{\ast }$%
-seminorms $\Gamma =\{p_{\lambda };\lambda \in \Lambda \}$. Then $A$ is
nuclear if only if the $C^{\ast }$-algebras $A_{\lambda },$ $\lambda \in
\Lambda $ are nuclear \cite[p. 126]{P}. Recall that a $C^{\ast }$-algebra $A$
is nuclear if and only if for any $C^{\ast }$-algebra $B$, the canonical map 
$k_{A,B}:A\otimes _{\max }B\rightarrow A\otimes _{\min }B,$ $k_{A,B}\left(
a\otimes b\right) =a\otimes b$ is a $C^{\ast }$-isomorphism.

\begin{remark}
\label{nuc} Let $(X,A)$ be a full Hilbert $C^{\ast }$-bimodule. If $A$ is
nuclear, then for any $C^{\ast }$-algebra $B$, the pair of maps $\left(
k_{X,B},k_{A,B}\right) $, where $k_{X,B}:X\otimes _{\max }B\rightarrow
X\otimes _{\min }B,$ $k_{X,B}\left( x\otimes b\right) =x\otimes b,$ is an
isomorphism of Hilbert $C^{\ast }$-bimodules and it induces a unique $%
C^{\ast }$-isomorphism $k:\left( A\otimes _{\max }B\right) \times _{X\otimes
_{\max }B}\mathbb{Z\rightarrow }\left( A\otimes _{\min }B\right) \times
_{X\otimes _{\min }B}\mathbb{Z}$ such that $k\circ i_{X\otimes _{\max
}B}=i_{X\otimes _{\min }B}\circ k_{X,B}$ and $k\circ i_{A\otimes _{\max
}B}=i_{A\otimes _{\min }B}\circ k_{A,B}.$
\end{remark}

\begin{proposition}
Let $(X,A)$ be a full Hilbert pro-$C^{\ast }$-bimodule. Then $A\times _{X}%
\mathbb{Z}$ is nuclear if and only if $A$ is nuclear.
\end{proposition}

\begin{proof}
According to \cite[Proposition 3.8]{JZ}, to prove this proposition it is
sufficient to show that if $(X,A)$ is a full Hilbert $C^{\ast }$-bimodule,
then $A\times _{X}\mathbb{Z}$ is nuclear if and only if $A$ is nuclear.

If $A$ is nuclear and $B$ a $C^{\ast }$-algebra, then $\Phi \circ k\circ
\Psi $, where $\Psi $ is the $C^{\ast }$-isomorphism from $\left( A\times
_{X}\mathbb{Z}\right) \otimes _{\max }B$ onto $\left( A\otimes _{\max
}B\right) \times _{X\otimes _{\max }B}\mathbb{Z}$ constructed in the proof
of Theorem \ref{tensor1}, $k$ is the $C^{\ast }$-isomorphism from $\left(
A\otimes _{\max }B\right) \times _{X\otimes _{\max }B}\mathbb{Z}$ onto $%
\left( A\otimes _{\min }B\right) $\ $\times _{X\otimes _{\min }B}\mathbb{Z}$
(Remark \ref{nuc}), and $\Phi $ is the $C^{\ast }$-isomorphism from $\left(
A\times _{X}\mathbb{Z}\right) \otimes _{\min }$ $B$ onto $\left( A\otimes
_{\min }B\right) \times _{X\otimes _{\min }B}\mathbb{Z}$ constructed in the
proof of Theorem \ref{tensor}, is a $C^{\ast }$-isomorphism. Moreover, $%
\left( \Phi \circ k\circ \Psi \right) \left( c\otimes b\right) =c\otimes b$
for all $c\in A\times _{X}\mathbb{Z}$ and for all $b\in B$. These show that, 
$A\times _{X}\mathbb{Z}$ is nuclear.

Conversely, suppose that $A\times _{X}\mathbb{Z}$ is nuclear. If $A\ $is not
nuclear, there is a $C^{\ast }$-algebra $B_{0}$ such that the canonical map $%
k_{A,B_{0}}:A\otimes _{\max }B_{0}\rightarrow A\otimes _{\min }B_{0}$ is not
injective. Then the $C^{\ast }$-morphism $k_{0}:\left( A\otimes _{\max
}B_{0}\right) \times _{X\otimes _{\max }B_{0}}\mathbb{Z\rightarrow }\left(
A\otimes _{\min }B_{0}\right) \times _{X\otimes _{\min }B_{0}}\mathbb{Z}$
induced by $\left( k_{X,B_{0}},k_{A,B_{0}}\right) $ is not injective.

On the other hand, because $A\times _{X}\mathbb{Z}$ is nuclear, the
canonical map $k_{A\times _{X}\mathbb{Z},B_{0}}:\left( A\times _{X}\mathbb{Z}%
\right) \otimes _{\max }B_{0}\rightarrow \left( A\times _{X}\mathbb{Z}%
\right) \otimes _{\min }$ $B_{0}\ $is a $C^{\ast }$-isomorphism, and then $%
\widetilde{k}=\Phi ^{-1}\circ k_{A\times _{X}\mathbb{Z},B_{0}}\circ \Psi
^{-1}\ $is a $C^{\ast }$-isomorphism from $\left( A\otimes _{\max
}B_{0}\right) \times _{X\otimes _{\max }B_{0}}\mathbb{Z}$ onto $\left(
A\otimes _{\min }B_{0}\right) \times _{X\otimes _{\min }B_{0}}\mathbb{Z}$.
Moreover, it is easy to check that $\widetilde{k}\circ i_{X\otimes _{\max
}B_{0}}=i_{X\otimes _{\min }B_{0}}\circ k_{X,B_{0}}$ and   $\widetilde{k}%
\circ i_{A\otimes _{\max }B_{0}}=i_{A\otimes _{\min }B_{0}}\circ k_{A,B_{0}}$
and then, by Remark \ref{nuc}, $\widetilde{k}=k_{0}$, a contradiction.
Therefore, if $A\times _{X}\mathbb{Z}$ is nuclear, then $A$ is nuclear.
\end{proof}

Another proof of the above proposition was given by Katsura (\cite[%
Proposition 7.8]{K}).

\textbf{Acknowledgements. }The author was partially supported by the grant
of the Romanian Ministry of Education, CNCS - UEFISCDI, project number
PN-II-ID-PCE-2012-4-0201.

\end{document}